\documentclass[11pt,a4paper]{article}
\usepackage[T1]{fontenc}
\usepackage[utf8]{inputenc}
\usepackage{authblk}
\usepackage{amssymb, amsmath}
\usepackage{amsthm, amsfonts,mathrsfs}
\usepackage[T1]{fontenc}
\usepackage{times}
\usepackage{graphicx}
\usepackage{subfigure}
\usepackage{cite}
\usepackage{hyperref}
\usepackage{epsfig}
\usepackage{color}
\usepackage[all]{xypic}
\hypersetup{colorlinks=true,citecolor=blue,linkcolor=blue}
\newtheorem{thm}{Theorem}[section]
\newtheorem{coro}[thm]{Corollary}

\newtheorem{lemma}[thm]{Lemma}
\newtheorem{prop}[thm]{Proposition}
\newtheorem{remark}[thm]{Remark}

\numberwithin{equation}{section}
\newcommand{\ep}{\epsilon}
\newcommand{\al}{\alpha}

\newcommand{\ue}{u^\epsilon}

\newcommand{\ve}{v^\epsilon}

\newcommand{\bwe}{\bar{w}^\epsilon}

\newcommand{\dbvf}{\dot{\overline{V}^\ep_1}}
\newcommand{\dbvs}{\dot{\overline{V}^\ep_2}}
\newcommand{\dbvt}{\dot{\overline{V}^\ep_3}}

\newcommand{\xe}{\xi^\epsilon}

\newcommand{\intot}{\int_{0}^{t}}
\newcommand{\intobtxl}{\int_{0}^{T}\int_{|x|<1}}
\newcommand{\intobtxg}{\int_{0}^{T}\int_{|x|\geq 1}}
\newcommand{\intos}{\int_{0}^{s}}

\newcommand{\supt}{\sup\limits_{0\leq t\leq T}}
\newcommand{\supe}{\sup\limits_{0<\epsilon\leq 1}}

\newcommand{\ube}{U^\epsilon}
\newcommand{\vbe}{V^\epsilon}

\newcommand{\bube}{\bar{U}^\epsilon}
\newcommand{\bvbe}{\bar{V}^\epsilon}

\newcommand{\re}{\rho^\epsilon}
\newcommand{\we}{w^\epsilon}

\title{\textbf{Convergence rate of   Smoluchowski--Kramers approximation with  stable L\'{e}vy noise }\thanks{This work is supported by NSFC Grant No.  12371243.}}

\author{Qingming Zhao\thanks{Corresponding author: qingming.zhao@smail.nju.edu.cn}~~}
\author{~Wei Wang\thanks{wangweinju@nju.edu.cn}}
\affil{School of Mathematics, Nanjing University, Nanjing 210093, P. R. China}

\date{} 

\begin{document}
	\maketitle

	\noindent{\small{\hspace{1.1cm} }}

	\noindent \textbf{Abstract~~~}   The small mass limit of the Langevin equation perturbed by $\al$-stable L\'{e}vy noise is considered by rewriting it in the form of slow-fast system, and spliting the fast component into three parts, where $\al\in(1,2).$ By exploring the three parts respectively,  the approximation equation is derived. The convergence is either in the sense of uniform metric or in the sense of Skorokhod metric, depending on how regular the noise is. In the former case, we obtain the convergence rate.
	\\[2mm]
	\textbf{Keywords~~~}{Langevin equation}, {Smoluchowski--Kramers approximation}, {Averaging}, {Singular pertubation}
	\\[2mm]
	\\
	\textbf{2020 Mathematics Subject Classification~~~}60G51, 60H10

	\section{Introduction}
	
	Smoluchowski--Kramers (SK for short) approximation is initially proposed by Smoluchowski\cite{SMOL16} and Kramers\cite{KRAM40} to derive an effective approximation to a Langevin equation which describes the motion of a particle with small mass. Roughly speaking,  the equation
	$$\ep \ddot{u^\ep}+\dot{u^\ep}=b(u^\ep)+\sigma(u^\ep)\dot{W}$$
	is  approximated, as $\ep\rightarrow 0$, in some sense by the equation 
	$$\dot{u}=b(u)+\sigma(u)\dot{W}.$$
	Formally, the limit equation is obtained by dropping the term $\ep \ddot{u}$\,. 
	
	There is fruitful work on SK approximation for Langevin equations with Gaussian white noise\cite[e.g.]{CX22, FH11, FH13, SW21, SW22, SWW24,LJW24}. The case that $W$ is an infinite dimensional Brownian motion is firstly studied by Cerrai and Freidlin\cite{CF06A, CF06B}. There is also some work concerned with SK approximation with colored noise which is highly oscillating in time \cite{HOTT13,HMV15}  or with L\'{e}vy noise \cite{ZHAN08,chaos24}. In this paper, we consider the following Langevin equation driven by a stable L\'{e}vy process 
	\begin{equation}
		\begin{cases}
			\ep\ddot{\ube}(t)+\dot{\ube}(t)=f(\ube(t))+\ep^\theta\dot{L}(t),\\
			\ube(0)=u_0, \quad \dot{\ube}(0)=v_0.
		\end{cases}\label{wave}
	\end{equation}
	Here, $0\leq\theta<1$ is a constant, and $L$ is an $\al$-stable process with $1<\al<2$, whose properties are detailed in Section~\ref{preli}. System (\ref{wave}) describes the motion of a particle with mass $\ep$ in an athermal fluctuation environment. Equation (\ref{wave}) can also be seen as a singularly perturbed    differential equation  with a random noise, which has attracted many researchers' interest~\cite[e.g.]{LW09, LRW11}.  Stable process is an important class of L\'{e}vy processes due to its self-similarity and scaling property\cite{KALL97, SATO99}. When the noise is a stable process, there is significant difference comparing to the Gaussian case. For example, stable process does not have finite second-order moment in general, and its L\'evy measure is infinite. Besides, the Garcia-Rademich-Rumsey theorem is invalid in the non-Gaussian case, which leads to remarkable difficulty when showing the tightness. 
	
	Formally,  the effective approximation model of (\ref{wave}) can also be obtained by dropping the $\ep \ddot{\ube}$ term, that is, 
	\begin{equation}
		\dot{\bar{U}}^\ep(t)=f(\bube(t))+\ep^\theta\dot{L}(t),\quad \bube(0)=u_0. \label{aveq}
	\end{equation}
	Obviously, the statement above reduces to the classical SK approximation when $\theta=0$. In this case, (\ref{aveq}) does not depend on $\ep$ and we write $\bar{U}$ instead of $\bube.$ Here we introduce a splitting technique of the solution \cite{LRW11} to show the approximation rigorously. Moreover, we also obtain the convergence rate.
	
	Rewrite the equation (\ref{wave}) as
	\begin{equation}
		\begin{cases}
			\dot{U}^\epsilon(t)=\vbe(t),\\
			\dot{V}^\epsilon(t)=\ep^{-1}[-\vbe(t)+f(\ube(t))]+\ep^{\theta-1}\dot{L}(t),\\
			\ube(0)=u_0,\quad  \vbe(0)=v_0.
		\end{cases}\label{slowfast}
	\end{equation}
	Equation (\ref{slowfast}) has a form of slow-fast system \cite{DW14}. Inspired by a splitting technique introduced by Lv et al.~\cite{LRW11}, we make the following important decomposition, which makes the analysis to (\ref{slowfast}) considerably more clear	
	\begin{equation}
		\begin{cases}
			\dbvf(t)=-\ep^{-1}\bvbe_1(t), \\
			\dbvs(t)=-\ep^{-1}[\bvbe_2(t)-f(\ube(t))],\\
			\dbvt(t)=-\ep^{-1}\bvbe_3(t)+\ep^{-\frac{1}{\al}}\dot{L}(t),\\
			\bvbe_1(0)=\ep v_0, \quad \bvbe_2(0)=0, \quad \bvbe_3(0)=0.
		\end{cases}\label{bveq}
	\end{equation}
	Direct calculation yields 
	\begin{equation}
		\vbe=\ep^{-1}\bvbe_1+\bvbe_2+\ep^{\theta+\frac{1}{\al}-1}\bvbe_3\,. \label{decomposition}
	\end{equation}

	The paper is organized as follows. In Section \ref{preli}, we impose some assumptions and state the main result. In Section \ref{severaltechnicallemma}, we give moment estimates and establish the tightness of $(\ube)$. After these preparation, we prove the main result in Section \ref{proofofmain}.
	\section{Preliminary and Main Result}\label{preli}
	Let $(\Omega,\mathcal{F},\mathbb{P})$ be a complete probability space, on which there is a filtration $(\mathcal{F}_t)_{0\leq t\leq T}$ satisfying the usual condition, where $0<T<\infty$ is fixed throughout the paper.  In the rest, $|x|:=\sqrt{\sum\limits_{i=1}^{d}x_i^2}$ for each $x=(x_1,...,x_d)\in\mathbb{R}^d,$ and $||A||:=\sup\limits_{x\in\mathbb{R}^d,|x|=1}|Ax|$ for each matrix $A\in\mathbb{R}^{d\times d}.$ Let $L$ be a L\'{e}vy process on $(\Omega,\mathcal{F}, (\mathcal{F}_t)_{0\leq t\leq T}, \mathbb{P}).$ Recall that the characteristic function of a rotation invariant $\al$-stable random vector $Z$ is
	$$\mathbb{E}e^{i(Z,h)}=e^{-c|h|^\al},$$
	for some $c>0$ \cite[Theorem 14.14]{SATO99}. An isotropic $\al$-stable L\'{e}vy process $L$ is a L\'{e}vy process such that $L(1)$ is a rotation invariant $\al$-stable random vector. In this case, 
	$$\mathbb{E}e^{i(L(t),h)}=e^{-ct|h|^\al},$$
	for each $h\in \mathbb{R}^d$ and $0\leq t\leq T$ \cite[Theorem 1.3.3]{APPL09}. Write $L=(L^1,...,L^d).$ A direct computation of characteristic function yields that each component $L^k$ is a one-dimensional isotropic $\al$-stable L\'evy process, and that they have the same distribution. It is well-known that for a stable random vector with L\'evy measure $\nu$, $$\int_{|x|\geq1}|x|^p\nu(dx)<\infty$$ if and only if $p<\al$ \cite[Example 25.10]{SATO99}. In this paper, we assume that $L$ is an isotropic $\al$-stable L\'evy process with $1<\al<2$. The L\'evy measure of $L$ is denoted by $\nu$. The L\'evy-It\^o decomposition of $L$ is written as
	\begin{equation}\label{lvdecom}
		L(t)=\intot\int_{|x|<1} x\tilde{N}(dsdx)+\intot\int_{|x|\geq1}xN(dsdx).
	\end{equation}
	The Poisson random measure, compensated Poisson random measure, and the L\'evy measure of the component $L_k$ is denoted by $N_k$, $\tilde{N}_k$ and $\nu_k$, respectively, $k=1,2,...d.$
	
	Let us recall some basic property of the Skorokhod space $\mathbb{D}:=\mathbb{D}([0,T];\mathbb{R}^d)$. The Skorokhod space $\mathbb{D}$ consists of all c\`adl\`ag $\mathbb{R}^d$-valued function on $[0,T].$ For $x,y\in\mathbb{D}$, set $$d^o(x,y):=\inf\limits_{\lambda\in\Lambda}\max\{||\lambda||^o, ||x-y\circ \lambda||_\infty\},$$
	where
	\begin{equation*}
		||\lambda||^o:=\sup\limits_{0\leq s\leq t}\Big|\log\frac{\lambda(t)-\lambda(s)}{t-s}\Big|,\quad ||f||_\infty:=\supt|f(t)|,
	\end{equation*} 
	 and $\Lambda$ consists of all strictly increasing continuous bijection on $[0,T].$ It is known that $d^o$ defines a complete separable metric on $\mathbb{D}$. For more details of the space $\mathbb{D}$, we refer to \cite[Chapter 3]{BIL13}.
	
	We make the following assumption.

	$(\mathbf{A})$ $f:\mathbb{R}^d\rightarrow\mathbb{R}^d$ is globally Lipschitz, that is, there exists a constant $L_f>0$ such that for all $x,y\in\mathbb{R}^d,$
	$$|f(x)-f(y)|\leq L_f|x-y|.$$
	
	In the followings, $C$ denotes constant whose value may change from line to line. Unless otherwise stated, the value of $C$ may depend on $T$ and the L\'{e}vy measure $\nu$, but it never depends on $\ep$. We use the notation $x\lesssim y$ to indicate that there exists a constant $C$ such that $x\leq Cy$. For two random elements $X$ and $Y$, we write $\mathcal{L}(X)$ to denote its distribution, and $X\overset{d}{=}Y$ means $\mathcal{L}(X)=\mathcal{L}(Y).$
	
	Our main result is the following theorem.
	
	\begin{thm}\label{mainresult}
		(i) Let $0\leq\theta<1.$ Under assumption $\mathbf{(A)}$,
		\begin{equation}
			\mathbb{E}\supt|\ube(t)-\bube(t)|\lesssim \ep^\theta.
		\end{equation} 	
		(ii) Let $\theta=0.$ Under assumption $\mathbf{(A)}$,
		\begin{equation}
			\lim\limits_{\ep\rightarrow 0}d^o(\ube,\bar{U})=0\quad\text{in probability}.
		\end{equation}
	\end{thm}
	\begin{remark}
		Obviously, part (i) of Theorem \ref{mainresult} does not give a convergence result in the case $\theta=0$. One might expect that there is the uniform convergence result
		$$\lim\limits_{\ep\to0}\mathbb{E}\supt|\ube(t)-\bar{U}(t)|=0.$$
		 However, we show that the above convergence does not hold at the end of Section \ref{proofofmain}.
	\end{remark}
	\section{Moment Estimate and Tightness of $(\ube)$}\label{severaltechnicallemma}
	In this section, we establish several moment estimates, which are essential to establish the tightness of $(\ube)$. We start with a well-posedness result. 
	\begin{lemma}\label{lemmaexistence}
		For each $0<\ep\leq 1$, the equation (\ref{slowfast}) admits a unique strong solution.
	\end{lemma}
	\begin{proof}
		Let $\psi^\ep:=\begin{pmatrix}
			\ube\\
			\vbe
		\end{pmatrix}.$ The equation (\ref{slowfast}) can be rewritten as
		\begin{equation}
			\dot{\psi^\ep}=\mathcal{A}^\ep\psi^\ep+\mathcal{F}^\ep(\psi^\ep)+\dot{\mathcal{L}}^\ep,\label{existence}
		\end{equation}
		where $\mathcal{A}^\ep:=\begin{pmatrix}
			0 & Id\\
			0 & -\ep^{-1}Id
		\end{pmatrix},$ 
		$\mathcal{F}^\ep\begin{pmatrix}
			u\\
			v
		\end{pmatrix}:=\begin{pmatrix}
			0\\
			\ep^{-1}f(u)
		\end{pmatrix}$
		and $\mathcal{L}^\ep:=\begin{pmatrix}
			0\\
			\ep^{\theta-1}L
		\end{pmatrix}.$
		Since $f$ is Lipschitz,  $\mathcal{A}^\ep+\mathcal{F}^\ep$ is also Lipschitz which leads to the existence and uniqueness~\cite[Theorem 6.2.9]{APPL09}.
	\end{proof}
	
	In order to give moment estimate for $\ube$, we first consider the linear part $\ue$ of (\ref{wave}), that is
	\begin{equation}\label{linearwave}
		\ep\ddot{\ue}+\dot{\ue}=\ep^\theta\dot{L},\quad \ue(0)=0, \dot{\ue}(0)=0.
	\end{equation}
	Similar to (\ref{slowfast}), we rewrite it as
	\begin{equation}\label{linearslowfast}
		\begin{cases}
			\dot{\ue}=\ve,\quad \ue(0)=0,\\
			\dot{\ve}=-\ep^{-1}\ve+\ep^{\theta-1}\dot{L},\quad\ve(0)=0.
		\end{cases}
	\end{equation}
	
	\begin{lemma}\label{thmforlinearEsup}
		$$\supe\mathbb{E}\supt|\ue(t)|<\infty.$$
	\end{lemma}
	\begin{proof}
		Set $\we$ to be the solution of the following linear SDE
		\begin{equation}\label{defofw}
			\dot{\we}=-\ep^{-1}\we+\ep^{-1/\al}\dot{L},\quad \we(0)=0,
		\end{equation} 
		 and it is straightforward to check that $\ve=\ep^{\theta+\frac{1}{\al}-1}\we,$ so
		$$\ue(t)=\intot\ve(s)ds=\ep^{\theta+\frac{1}{\al}-1}\intot\we(s)ds.$$
		By the definition of $\we$, we have
		$$\we(t)=-\ep^{-1}\intot\we(s)ds+\ep^{-1/\al}L(t).$$
		Multiplying $\ep^{\theta+\frac{1}{\al}}$ and rearranging,
		$$\ep^{\theta+\frac{1}{\al}-1}\intot\we(s)ds=\ep^{\theta}(L(t)-\ep^{\frac{1}{\al}}\we(t)),$$
		so that 
		$$\mathbb{E}\supt|\ue(t)|=\ep^{\theta}\mathbb{E}\supt|L(t)-\ep^{\frac{1}{\al}}\we(t)|.$$
		Set $\bwe:=\ep^{\frac{1}{\al}}\we,$ and one verifies immediately that
		$$d\bwe(t)=-\ep^{-1}\bwe(t)+dL(t),\quad \bwe(0)=0.$$
		The proposition is proved provided that we show
		\begin{equation}\label{esforbwe}
			\supe\mathbb{E}\supt|\bwe(t)|<\infty,
		\end{equation}
		 and
		 \begin{equation}\label{esforL}
		 	\mathbb{E}\supt|L(t)|<\infty.
		 \end{equation}
		Let us show (\ref{esforbwe}) first. Indeed, applying It\^{o}'s formula for $\phi(x)=(|x|^2+1)^{p/2},$ and note that
		\begin{equation}\label{phiproperty}
			|D\phi(y)|\lesssim|y|^{p-1},\quad ||D^2\phi(y)||\lesssim1,
		\end{equation}
		\begin{eqnarray}\label{esforbv4}
			&&\phi(\bwe(t))\nonumber\\&=&1+\intot (D\phi(\bwe(s)),-\ep^{-1}\bwe(s))ds\nonumber\\&&{}+\intot\int_{|x|<1}\phi(\bwe(s-)+x)-\phi(\bwe(s-))\tilde{N}(dsdx)\nonumber\\&&{}+\intot\int_{|x|\geq1}\phi(\bwe(s-)+x)-\phi(\bwe(s-))N(dsdx)\nonumber\\&&{}+\intot\int_{|x|<1}\phi(\bwe(s)+x)-\phi(\bwe(s))-(D\phi(\bwe(s)),x)\nu(dx)ds\nonumber\\&=:&1+\sum_{k=1}^{4}H^\ep_k(t).
		\end{eqnarray}
		Since $$D\phi(y)=\frac{py}{(|y|^2+1)^{1-p/2}},$$ one immediately obtains $$H^\ep_1(t)=\intot\frac{-p|\bwe(s)|^2}{\ep(|\bwe(s)|^2+1)^{1-p/2}}ds\leq 0,$$
		and 
		\begin{equation}
			\mathbb{E}\supt H^\ep_1(t)\leq 0.\label{esforH1}
		\end{equation}
		By Burkholder--Davis--Gundy inequality\cite[Theorem 3.50]{peszat2007stochastic}, Taylor formula and (\ref{phiproperty}),
		\begin{eqnarray}\label{esforH2}
			&&\mathbb{E}\supt H^\ep_2(t)\nonumber\\
			&\lesssim&\mathbb{E}\sqrt{\intobtxl|\phi(\bwe(s-)+x)-\phi(\bwe(s-))|^2N(dsdx)}\nonumber \\
			&\leq&\sqrt{\mathbb{E}\intobtxl|\phi(\bwe(s-)+x)-\phi(\bwe(s-))|^2N(dsdx)}\nonumber \\&=&\sqrt{\mathbb{E}\intobtxl|\phi(\bwe(s)+x)-\phi(\bwe(s))|^2\nu(dx)ds}\nonumber \\&=&\sqrt{\mathbb{E}\intobtxl|(D\phi(\bwe(s)+\theta x),x)|^2\nu(dx)ds}\nonumber\\&\lesssim&\sqrt{\mathbb{E}\intobtxl|\bwe(s)+\theta x|^{2p-2}|x|^2\nu(dx)ds}\nonumber \\&\leq&\sqrt{\mathbb{E}\intobtxl|\bwe(s)|^{2p-2}|x|^2\nu(dx)ds+\intobtxl|x|^{2p}\nu(dx)ds}\nonumber\\
			&\leq&\mathbb{E}\intobtxl|\bwe(s)|^{2p-2}|x|^2\nu(dx)ds+\intobtxl|x|^{2p}\nu(dx)ds+1 \nonumber \\
			&\lesssim&\int_{0}^{T}\mathbb{E}|\bwe(s)|^{2p-2}ds+1\nonumber
			\\&\leq &\int_{0}^{T}\mathbb{E}\sup\limits_{0\leq s\leq t}\phi(\bwe(s))dt+1	 
		\end{eqnarray}
		where we have used  $\sqrt{z}\leq z+1$ and $2p-2\leq p$.

		For $H^\ep_3$\,,   we have 
		\begin{eqnarray*}
			&&\mathbb{E}\supt H^\ep_3(t)\nonumber\\&\leq&\mathbb{E}\intobtxg|\phi(\bwe(s-)+x)-\phi(\bwe(s-))|N(dsdx)\nonumber\\&=&\mathbb{E}\intobtxg|\phi(\bwe(s)+x)-\phi(\bwe(s))|\nu(dx)ds.
		\end{eqnarray*}
		By Taylor formula, Young inequality and (\ref{phiproperty}),
		\begin{eqnarray*}
			&&\intobtxg|\phi(\bwe(s)+x)-\phi(\bwe(s))|\nu(dx)ds\\&=&\intobtxg|(D\phi(\we(s)+\theta x),x)|\nu(dx)ds\\&\leq&\intobtxg|D\phi(\we(s)+\theta x)|\cdot|x|\nu(dx)ds\\&\lesssim&\intobtxg|\we(s)+\theta x|^{p-1}|x|\nu(dx)ds\\&\leq&\intobtxg(|\we(s)|^{p-1}+|x|^{p-1})|x|\nu(dx)ds\\&\lesssim&\int_0^T\phi(\we(s))ds+1,
		\end{eqnarray*}
		so that
		\begin{eqnarray}\label{esforH3}
			\mathbb{E}\supt H^\ep_3(t)\lesssim\int_{0}^{T}\mathbb{E}\phi(\bwe(s))ds+1\leq \int_{0}^{T}\mathbb{E}\sup\limits_{0\leq s\leq t}\phi(\bwe(s))dt+1.
		\end{eqnarray}
		Lastly we turn to $H^\ep_4$\,.  By Taylor formula and (\ref{phiproperty}),
		\begin{eqnarray*}
			&&\supt H^\ep_4(t)\nonumber\\&\leq &\intobtxl|\phi(\bwe(s)+x)-\phi(\bwe(s))-(D\phi(\bwe(s)),x)|\nu(dx)ds\nonumber\\&\leq &\intobtxl|D^2\phi(\bwe(s)+\theta x)(x\otimes x)|\nu(dx)ds\nonumber\\&\lesssim&\intobtxl|x|^2\nu(dx)ds\nonumber\\&\lesssim&1,
		\end{eqnarray*}
		which means
		\begin{equation}
			\mathbb{E}\supt H^\ep_4(t)\lesssim 1.\label{esforH4}
		\end{equation}
		Taking supremum and expectation in (\ref{esforbv4}) and using (\ref{esforH1})--(\ref{esforH4}),
		$$\mathbb{E}\supt\phi(\bwe(t))\lesssim \int_{0}^{T}\mathbb{E}\sup\limits_{0\leq s\leq t}\phi(\bwe(s))dt+1,$$
		and we end our proof of (\ref{esforbwe}) by Gronwall's inequality. The proof of (\ref{esforL}) is standard. In fact, by Burkholder-Davis-Gundy inequality,
		\begin{eqnarray}
			&&\mathbb{E}\supt|L(t)|\nonumber\\&\leq&\mathbb{E}\supt\Big|\intot\int_{|x|<1} x\tilde{N}(dsdx)\Big|+\mathbb{E}\supt\Big|\intot\int_{|x|\geq 1} x{N}(dsdx)\Big|\nonumber\\&\leq&\mathbb{E}\sqrt{\int_{0}^{T}\int_{|x|<1}x^2 {N}(dsdx)}+\mathbb{E}\int_{0}^{T}\int_{|x|\geq1}|x|N(dsdx)\nonumber\\&\leq&\sqrt{\mathbb{E}\int_{0}^{T}\int_{|x|<1}x^2 N(dsdx)}+\mathbb{E}\int_{0}^{T}\int_{|x|\geq1}|x|N(dsdx)\nonumber\\&=&\sqrt{T\int_{|x|<1}x^2\nu(dx)}+T\int_{|x|\geq1}|x|\nu(dx)\nonumber\\&\lesssim&1.
		\end{eqnarray}
		This finishes the proof.
	\end{proof}
	We give the moment estimate for $\ube$ with the help of Lemma \ref{thmforlinearEsup}.
	\begin{prop}\label{thmforEsup}
		$$\supe\mathbb{E}\supt|\ube(t)|<\infty.$$
	\end{prop}
	\begin{proof}
		Set $\re:=\ube-\ue.$ It follows from (\ref{wave}) and (\ref{linearwave}) that
		\begin{equation}\label{rde}
			\ep\ddot{\re}+\dot{\re}=f(\ue+\re),\quad \re(0)=u_0,\dot{\re}(0)=v_0,
		\end{equation}
		and we can rewrite it as
		\begin{equation}\label{rdeslowfast}
			\begin{cases}
				\dot{\re}=\xe,\re(0)=u_0\\
				\dot{\xe}=\ep^{-1}(-\xe+f(\ue+\re)), \xe(0)=v_0.
			\end{cases}
		\end{equation}
		From (\ref{rdeslowfast}), we can solve $\xe$ analytically as
		$$\xi(t)=\ep^{-1}e^{-\ep^{-1}t}\intot e^{\ep^{-1}s}(f(\ue(s)+\re(s)))ds,$$
		so that by the Lipschitz property of $f$,
		\begin{eqnarray*}
			&&|\xe(t)|\\&\leq&\ep^{-1}e^{-\ep^{-1}t}\intot e^{\ep^{-1}s}|f(\ue(s)+\re(s))|ds\\&\lesssim&\ep^{-1}e^{-\ep^{-1}t}\intot e^{\ep^{-1}s}|\ue(s)|ds+\ep^{-1}e^{-\ep^{-1}t}\intot e^{\ep^{-1}s}|\re(s)|ds+\ep^{-1}e^{-\ep^{-1}t}\intot e^{\ep^{-1}s}ds\\&\leq&\ep^{-1}e^{-\ep^{-1}t}\intot e^{\ep^{-1}s}\supt|\ue(t)|ds+\ep^{-1}e^{-\ep^{-1}t}\intot e^{\ep^{-1}s}|\re(s)|ds+1\\&\leq&\supt|\ue(t)|+\ep^{-1}e^{-\ep^{-1}t}\intot e^{\ep^{-1}s}|\re(s)|ds+1.
		\end{eqnarray*}
		Since $\dot\re=\xe,$ we can continue our estimate as
		\begin{eqnarray*}
			&&|\re(t)|\\&\leq&\intot|\xe(s)|ds\\&\lesssim&t\supt|\ue(t)|+\intot\ep^{-1}e^{-\ep^{-1}s}\intos e^{\ep^{-1}r}|\re(r)|drds+1\\&=&t\supt|\ue(t)|+\ep^{-1}\intot\intos e^{-\ep^{-1}(s-r)}|\re(r)|drds+1\\&=&t\supt|\ue(t)|+\ep^{-1}\intot\int_{r}^{t} e^{-\ep^{-1}(s-r)}|\re(r)|dsdr+1\\&=&t\supt|\ue(t)|+\ep^{-1}\intot e^{\ep^{-1}r}|\re(r)|\int_{r}^{t}e^{-\ep^{-1}s}dsdr+1\\&=&t\supt|\ue(t)|+\ep^{-1}\intot e^{\ep^{-1}r}|\re(r)|(\ep e^{-\ep^{-1}r}-\ep e^{-\ep^{-1}t})dr+1\\&\leq&t\supt|\ue(t)|+\intot|\re(r)|dr+1.
		\end{eqnarray*}
		Taking supremum with respect to $t$,
		\begin{eqnarray*}
			&&\supt|\re(t)|\\&\lesssim&T\supt|\ue(t)|+\int_0^T|\re(r)|dr+1\\&\lesssim&T\supt|\ue(t)|+\int_0^T\sup_{0\leq s\leq t}|\re(s)|dt+1.
		\end{eqnarray*}
		Taking expectation and applying Lemma \ref{thmforlinearEsup},
		\begin{eqnarray*}
			\mathbb{E}\supt|\re(t)|\lesssim\int_0^T\mathbb{E}\sup_{0\leq s\leq t}|\re(s)|dt+1.
		\end{eqnarray*}
		By Gronwall inequality,
		\begin{equation}\label{esforre}
			\mathbb{E}\supt|\re(t)|\lesssim1.
		\end{equation}
		Since $\ube=\ue+\re,$ our proof is finished by combining (\ref{esforre}) and Lemma \ref{thmforlinearEsup}.
	\end{proof}
	Utilizing the uniform moment estimate, we establish the tightness of $(\ube)_{0<\ep\leq1}.$ 
	\begin{prop}\label{thmfortight}
		$(\mathcal{L}(\ube))_{0<\ep\leq 1}$ is tight on $\mathbb{D}([0,T];\mathbb{R}^d).$
	\end{prop}
	\begin{remark}\label{nottightinC}
		One finds that, since $\dot{\ube}=\vbe$, $\ube$ is the Lebesgue integral of $\vbe$, and in particular the trajectory of $\ube$ belongs to $\mathbb{C}:=\mathbb{C}([0,T];\mathbb{R}^d).$ One might try to show the tightness of $(\ube)$ in $\mathbb{C},$ but at least when $\theta=0$, that is, in the case of classical SK approximation, it is impossible. To see this, suppose by contradictory that $(\ube)$ is tight in $\mathbb{C}.$ By the same argument proving part (ii) of Theorem \ref{mainresult} appearing in Section \ref{proofofmain}, one shows that $(\ube)$ converges in probability in the space $\mathbb{C}$ to its limit $U$, and that $$dU(t)=f(U(t))dt+dL(t).$$ Passing to a subsequence, $U\in\mathbb{C}$ a.s.. However, if we take $f=0$ and $u_0=0$, then $U=L$, and that is to say $L\in\mathbb{C}$ a.s., which is not true.
	\end{remark}
	\begin{proof}
		We show the proposition by checking the Aldous tightness criterion \cite[Theorem 16.10]{BIL13}. By Chebyshev inequality, it is sufficient to verify that
		
		(i) $\supe\mathbb{E}\supt|\ube(t)|<\infty,$
		
		(ii) There exists a map $f: \mathbb{R}_+\to \mathbb{R}_+$ with $\lim\limits_{\delta\downarrow0}|f(\delta)|\downarrow 0$ such that for each $\ep>0,$ $\delta>0$ and stopping time $\tau$ with $\tau+\delta\leq T,$
		$$\mathbb{E}|\ube(\tau+\delta)-\ube(\tau)|\leq f(\delta).$$
		
		Note that (i) is precisely the Proposition \ref{thmforEsup}, so we only need to check (ii).
		
		From (\ref{slowfast}) we see that $\vbe$ is an Ornstein-Uhlenbeck process driven by a semimartingale $$dZ^\ep(t)=\ep^{-1}f(\ube(t))dt+\ep^{\theta-1}dL(t),$$ so that
		$$\vbe(t)=e^{-\ep^{-1}t}v_0+\ep^{-1}\intot e^{-\ep^{-1}(t-s)}f(\ube(s))ds+\ep^{\theta-1}\intot e^{-\ep^{-1}(t-s)}dL(s).$$
		Since $\dot{\ube}=\vbe,$ we have
		\begin{eqnarray}
			&&\ube(t)\nonumber\\&=&\intot e^{-\ep^{-1}s}ds\cdot v_0+\ep^{-1}\intot\intos e^{-\ep^{-1}(s-r)}f(\ube(r))drds+\ep^{\theta-1}\intot\intos e^{-\ep^{-1}(s-r)}dL(r)ds\nonumber\\&=:&I^\ep_1(t)+I^\ep_2(t)+I^\ep_3(t).
		\end{eqnarray}
		Let us deal with $I^\ep_1$ first. By the elementary inequality $1-e^{-x}\leq x,$
		\begin{eqnarray}
			&&|I^\ep_1(\tau+\delta)-I^\ep_1(\tau)|\nonumber\\&=&|v_0|\Big|\int_{\tau}^{\tau+\delta}e^{-\ep^{-1}s}ds\Big|\nonumber\\&=&|v_0|\ep e^{-\ep^{-1}\tau}(1-e^{-\ep^{-1}\delta})\nonumber\\&\leq&|v_0|\ep e^{-\ep^{-1}\tau}\ep^{-1}\delta\nonumber\\&=&|v_0|e^{-\ep^{-1}\tau}\delta\nonumber\\&\leq&|v_0|\delta,
		\end{eqnarray}
		so \begin{equation}\label{tightI1}
			\mathbb{E}|I^\ep_1(\tau+\delta)-I^\ep_1(\tau)|\lesssim\delta.
		\end{equation}
		By Fubini theorem,
		\begin{eqnarray}\label{tightI2ready}
			&&I^\ep_2(t)\nonumber\\&=&\ep^{-1}\intot\int_r^t e^{-\ep^{-1}(s-r)}f(\ube(r))dsdr\nonumber\\&=&\ep^{-1}\intot e^{\ep^{-1}r}f(\ube(r))\int_r^te^{-\ep^{-1}s}dsdr\nonumber\\&=&\ep^{-1}\intot e^{\ep^{-1}r}f(\ube(r))\ep(e^{-\ep^{-1}r}-e^{-\ep^{-1}t})dr\nonumber\\&=&\intot f(\ube(r))dr-e^{-\ep^{-1}t}\intot e^{\ep^{-1}r}f(\ube(r))dr\nonumber\\&=:&I^\ep_{21}(t)-I^\ep_{22}(t).
		\end{eqnarray}
		Since $f$ is Lipschitz,
		\begin{eqnarray}\label{tightI21}
			&&|I^\ep_{21}(\tau+\delta)-I^\ep_{21}(\tau)|\nonumber\\&=&\Big|\int_{\tau}^{\tau+\delta}f(\ube(r))dr\Big|\nonumber\\&\lesssim&\int_{\tau}^{\tau+\delta}|\ube(r)|+1dr\nonumber\\&\leq&\supt(|\ube(t)|+1)\delta.
		\end{eqnarray}
		\begin{eqnarray}\label{tightI22}
			&&|I^\ep_{22}(\tau+\delta)-I^\ep_{22}(\tau)|\nonumber\\&=&\Big|e^{-\ep^{-1}(\tau+\delta)}\int_{0}^{\tau+\delta}e^{\ep^{-1}r}f(\ube(r))dr-e^{-\ep^{-1}\tau}\int_{0}^{\tau}e^{\ep^{-1}r}f(\ube(r))dr\Big|\nonumber\\&\leq&\Big|e^{-\ep^{-1}(\tau+\delta)}\int_{0}^{\tau+\delta}e^{\ep^{-1}r}f(\ube(r))dr-e^{-\ep^{-1}(\tau+\delta)}\int_{0}^{\tau}e^{\ep^{-1}r}f(\ube(r))dr\Big|\nonumber\\&&+\Big|e^{-\ep^{-1}(\tau+\delta)}\int_{0}^{\tau}e^{\ep^{-1}r}f(\ube(r))dr-e^{-\ep^{-1}\tau}\int_{0}^{\tau}e^{\ep^{-1}r}f(\ube(r))dr\Big|\nonumber\\&=&\Big|e^{-\ep^{-1}(\tau+\delta)}\int_{\tau}^{\tau+\delta}e^{\ep^{-1}r}f(\ube(r))dr\Big|+e^{-\ep^{-1}\tau}(1-e^{\ep^{-1}\delta})\Big|\int_{0}^{\tau}e^{\ep^{-1}r}f(\ube(r))dr\Big|\nonumber\\&\leq&\int_{\tau}^{\tau+\delta}|f(\ube(r))|dr+e^{-\ep^{-1}\tau}\frac{\delta}{\ep}\int_{0}^{\tau}e^{\ep^{-1}r}|f(\ube(r))|dr\nonumber\\&\lesssim&\int_{\tau}^{\tau+\delta}\supt|\ube(t)|+1dr+e^{-\ep^{-1}\tau}\frac{\delta}{\ep}\int_{0}^{\tau}e^{\ep^{-1}r}(\supt|\ube(t)|+1)dr\nonumber\\&=&(\supt|\ube(t)|+1)\delta+(\supt|\ube(t)|+1)e^{-\ep^{-1}\tau}\frac{\delta}{\ep}\int_{0}^{\tau}e^{\ep^{-1}r}dr\nonumber\\&\leq&(\supt|\ube(t)|+1)\delta+(\supt|\ube(t)|+1)e^{-\ep^{-1}\tau}\frac{\delta}{\ep}\ep e^{\ep^{-1}\tau}\nonumber\\&=&2(\supt|\ube(t)|+1)\delta.
		\end{eqnarray}
		By (\ref{tightI2ready})-(\ref{tightI22}) and Proposition \ref{thmforEsup},
		\begin{equation}\label{tightI2}
			\mathbb{E}|I^\ep_{2}(\tau+\delta)-I^\ep_{2}(\tau)|\lesssim(\mathbb{E}\supt|\ube(t)|+1)\delta\lesssim\delta.
		\end{equation}
		\begin{eqnarray}\label{tightI3ready}
			&&I^\ep_3(t)\nonumber\\&=&\intot\intos\ep^{\theta-1} e^{-\ep^{-1}(s-r)}dL(r)ds\nonumber\\&=&\sum\limits_{k=1}^{d}\intot\intos\ep^{\theta-1} e^{-\ep^{-1}(s-r)}dL_k(r)ds\cdot e_k\nonumber\\&=:&\sum\limits_{k=1}^{d}J^\ep_k(t)e_k,
		\end{eqnarray}
		where $(e_k)$ is the orthonomal basis of $\mathbb{R}^d.$ By the L\'{e}vy-It\^{o} decomposition,
		\begin{eqnarray}\label{tightJkready}
			&&J^\ep_k(t)\nonumber\\&=&\ep^{\theta-1}\intot\intos\int_{|x|\geq 1}e^{-\frac{1}{\ep}(s-r)}xN_k(drdx)ds\nonumber\\&&+\ep^{\theta-1}\intot\intos\int_{|x|< 1}e^{-\frac{1}{\ep}(s-r)}x\tilde{N}_k(drdx)ds\nonumber\\&=:&J^1_k(t)+J^2_k(t).
		\end{eqnarray}
		First we deal with $J^1_k.$ Recall the integration-by-part formula for real semimartingale(\cite[Page 68]{PRO05}), that is, $$\intot X(s-)dY(s)=X(t)Y(t)-X(0)Y(0)-\intot Y(s-)dX(s)-[X,Y](t),$$ and therefore, 
		\begin{eqnarray}\label{tightJ1kready}
			&&J^1_k(t)\nonumber\\&=&\ep^\theta\intot\ep^{-1}e^{-\ep^{-1}s}\intos\int_{|x|\geq 1}e^{\ep^{-1}r}xN_k(drdx)ds\nonumber\\&=&\ep^\theta\intot\intos\int_{|x|\geq 1}e^{\ep^{-1}r}xN_k(drdx)d(-e^{-\ep^{-1}s})\nonumber\\&=&\ep^\theta\Big[-\intot\int_{|x|\geq 1}e^{\ep^{-1}r}xN_k(drdx)e^{-\ep^{-1}t}+\intot e^{-\ep^{-1}s}d\intos\int_{|x|\geq 1}e^{\ep^{-1}r}xN_k(drdx)\Big]\nonumber\\&=&\ep^\theta\Big[-\intot\int_{|x|\geq 1}e^{\ep^{-1}r}xN_k(drdx)e^{-\ep^{-1}t}+\intot\int_{|x|\geq 1}xN_k(dsdx)\Big]\nonumber\\&=:&\ep^\theta(J^{11}_k(t)+J^{12}_k(t)).
		\end{eqnarray}
		For $J^{11}_k,$
		\begin{eqnarray}
			&&\mathbb{E}|J^{11}_k(\tau+\delta)-J^{11}_k(\tau)|\nonumber\\&=&\mathbb{E}\Big|-e^{-\ep^{-1}(\tau+\delta)}\int_{0}^{\tau+\delta}\int_{|x|\geq 1}e^{\ep^{-1}r}xN_k(drdx)+e^{-\ep^{-1}\tau}\int_{0}^{\tau}\int_{|x|\geq 1}e^{\ep^{-1}r}xN_k(drdx)\Big|\nonumber\\&\leq&\mathbb{E}\Big|-e^{-\ep^{-1}(\tau+\delta)}\int_{0}^{\tau+\delta}\int_{|x|\geq 1}e^{\ep^{-1}r}xN_k(drdx)+e^{-\ep^{-1}(\tau+\delta)}\int_{0}^{\tau}\int_{|x|\geq 1}e^{\ep^{-1}r}xN_k(drdx)\Big|\nonumber\\&&+\mathbb{E}\Big|e^{-\ep^{-1}(\tau+\delta)}\int_{0}^{\tau}\int_{|x|\geq 1}e^{\ep^{-1}r}xN_k(drdx)-e^{-\ep^{-1}\tau}\int_{0}^{\tau}\int_{|x|\geq 1}e^{\ep^{-1}r}xN_k(drdx)\Big|\nonumber\\&=:&A+B.
		\end{eqnarray}
		By a direct computation,
		\begin{eqnarray}
			&&A\nonumber\\&=&\mathbb{E}\Big|-e^{-\ep^{-1}(\tau+\delta)}\int_{\tau}^{\tau+\delta}\int_{|x|\geq 1}e^{\ep^{-1}r}xN_k(drdx)\Big|\nonumber\\&\leq&\mathbb{E}\Big[-e^{-\ep^{-1}(\tau+\delta)}\int_{\tau}^{\tau+\delta}\int_{|x|\geq 1}|e^{\ep^{-1}r}x|N_k(drdx)\Big]\nonumber\\&\leq&\mathbb{E}\Big[\int_{\tau}^{\tau+\delta}\int_{|x|\geq 1}|x|N_k(drdx)\Big]\nonumber\\&=&\mathbb{E}\Big[\int_{\tau}^{\tau+\delta}\int_{|x|\geq 1}|x|\nu_k(dx)dr\Big]\nonumber\\&=&\mathbb{E}\Big[\int_{0}^{T}\int_{|x|\geq 1}|x|\mathbb{I}_{[\tau,\tau+\delta]}(r)\nu_1(dx)dr\Big]\nonumber\\&=&\int_{0}^{T}\int_{|x|\geq 1}|x|\mathbb{E}\mathbb{I}_{[\tau,\tau+\delta]}(r)\nu_1(dx)dr\nonumber\\&=&\int_{|x|\geq 1}|x|\nu_1(dx)\int_{0}^{T}\mathbb{E}\mathbb{I}_{[\tau,\tau+\delta]}(r)dr\nonumber\\&=&\int_{|x|\geq 1}|x|\nu_1(dx)\delta\nonumber\\&\lesssim&\delta,
		\end{eqnarray}
		and
		\begin{eqnarray}
			&&B\nonumber\\&=&\mathbb{E}\Big|(e^{-\ep^{-1}\tau}-e^{-\ep^{-1}(\tau+\delta)})\int_{0}^{\tau}\int_{|x|\geq 1}e^{\ep^{-1}r}xN_k(drdx)\Big|\nonumber\\&=&\mathbb{E}\Big|e^{-\ep^{-1}\tau}(1-e^{-\ep^{-1}\delta})\int_{0}^{\tau}\int_{|x|\geq 1}e^{\ep^{-1}r}xN_k(drdx)\Big|\nonumber\\&\leq&\frac{\delta}{\ep}\mathbb{E}\Big|e^{-\ep^{-1}\tau}\int_{0}^{\tau}\int_{|x|\geq 1}e^{\ep^{-1}r}xN_k(drdx)\Big|\nonumber\\&\leq&\frac{\delta}{\ep}\mathbb{E}\Big[\int_{0}^{\tau}\int_{|x|\geq 1}e^{-\ep^{-1}(\tau-r)}|x|N_k(drdx)\Big]\nonumber\\&\leq&\frac{\delta}{\ep}\mathbb{E}\Big[\int_{0}^{\tau}\int_{|x|\geq 1}e^{-\ep^{-1}(\tau-r)}|x|\nu_1(dx)dr\Big]\nonumber\\&=&\frac{\delta}{\ep}\mathbb{E}\Big[\int_{0}^{T}\int_{|x|\geq 1}e^{-\ep^{-1}(\tau-r)}|x|\mathbb{I}_{[0,\tau]}(r)\nu_1(dx)dr\Big]\nonumber\\&=&\frac{\delta}{\ep}\int_{|x|\geq 1}|x|\nu_1(dx)\int_{0}^{T}\mathbb{E}(\mathbb{I}_{[0,\tau]}(r)e^{-\ep^{-1}(\tau-r)})dr\nonumber\\&=&\frac{\delta}{\ep}\int_{|x|\geq 1}|x|\nu_1(dx)\mathbb{E}(e^{-\ep^{-1}\tau}\int_{0}^{\tau}e^{\ep^{-1}r}dr)\nonumber\\&\lesssim&\delta.
		\end{eqnarray}
		By the two estimates above, we have
		\begin{equation}\label{tightJ11k}
			\mathbb{E}|J^{11}_k(\tau+\delta)-J^{11}_k(\tau)|\lesssim\delta.
		\end{equation}
		
		For $J^{12}_k, $by Fubini theorem,
		\begin{eqnarray}\label{tightJ12k}
			&&\mathbb{E}|J^{12}_k(\tau+\delta)-J^{12}_k(\tau)|\nonumber\\&=&\mathbb{E}\Big|\int_{\tau}^{\tau+\delta}\int_{|x|\geq 1}xN_k(dsdx)\Big|\nonumber\\&\leq&\mathbb{E}\int_{\tau}^{\tau+\delta}\int_{|x|\geq 1}|x|N_k(dsdx)\nonumber\\&=&\mathbb{E}\int_{\tau}^{\tau+\delta}\int_{|x|\geq 1}|x|\nu_k(dx)ds\nonumber\\&=&\mathbb{E}\int_{0}^{T}\int_{|x|\geq 1}\mathbb{I}_{[\tau,\tau+\delta]}(s)|x|\nu_1(dx)ds\nonumber\\&=&\int_{0}^{T}\int_{|x|\geq1}\mathbb{E}(\mathbb{I}_{[\tau,\tau+\delta]}(s)|x|)\nu_1(dxds)\nonumber\\&=&\int_{|x|\geq 1}|x|\nu_1(dx)\int_{0}^{T}\mathbb{E}\mathbb{I}_{[\tau,\tau+\delta]}(s)ds\nonumber\\&=&\int_{|x|\geq 1}|x|\nu_1(dx)\mathbb{E}\int_{0}^{T}\mathbb{I}_{[\tau,\tau+\delta]}(s)ds\nonumber\\&\lesssim&\delta.
		\end{eqnarray}
		Combining (\ref{tightJ11k}) and (\ref{tightJ12k}), we have
		\begin{equation}\label{tightJ1k}
			\mathbb{E}|J^1_k(\tau+\delta)-J^1_k(\tau)|\lesssim\ep^{\theta}\delta\leq\delta.
		\end{equation}
		We turn to deal with $J^2_k.$ For notation simplicity, we set $\tilde{L}_k(t):=\intot\int_{|x|< 1}x\tilde{N}_k(dsdx),$ which means $$J^2_k(t)=\intot\int_{0}^{u}\ep^{\theta-1}e^{-\ep^{-1}u}e^{\ep^{-1}h}d\tilde{L}_k(h)du,$$ so that
		$$J^2_k(t)-J^2_k(s)=\int_{s}^{t}\int_{0}^{u}\ep^{\theta-1}e^{-\ep^{-1}u}e^{\ep^{-1}h}d\tilde{L}_k(h)du.$$
		By stochastic Fubini theorem(\cite[Theorem 64]{PRO05}),
		$$J^2_k(t)-J^2_k(s)=\Big[\intos\int_{s}^{t}\ep^{\theta-1}e^{-\ep^{-1}u}e^{\ep^{-1}h}dud\tilde{L}_k(h)+\int_{s}^{t}\int_{h}^{t}\ep^{\theta-1}e^{-\ep^{-1}u}e^{\ep^{-1}h}dud\tilde{L}_k(h)\Big].$$
		Therefore,
		\begin{eqnarray}\label{tightJ2kready}
			&&J^2_k(\tau+\delta)-J^2_k(\tau)\nonumber\\&=&\Big[\int_{0}^{\tau}\int_{\tau}^{\tau+\delta}\ep^{\theta-1}e^{-\ep^{-1}u}e^{\ep^{-1}h}dud\tilde{L}_k(h)+\int_{\tau}^{\tau+\delta}\int_{h}^{\tau+\delta}\ep^{\theta-1}e^{-\ep^{-1}u}e^{\ep^{-1}h}dud\tilde{L}_k(h)\Big]\nonumber\\&=:& J^{21}_k+J^{22}_{k}.
		\end{eqnarray}
		By the definition of $\tilde{L}$, It\^{o} isometry, and H\"older inequality,
		\begin{eqnarray}\label{geforj2k}
			&&\mathbb{E}|J^{21}_k|^2\nonumber\\&=&\mathbb{E}\int_{0}^{\tau}\int_{|x|< 1}\Big|\int_{\tau}^{\tau+\delta}x\ep^{\theta-1}e^{-\ep^{-1}u}e^{\ep^{-1}h}du\Big|^2\nu_1(dx)dh\nonumber\\&=&\int_{|x|< 1}|x|^2\nu_1(dx)\mathbb{E}\int_{0}^{\tau}\Big|\int_{\tau}^{\tau+\delta}\ep^{\theta-1}e^{-\ep^{-1}u}e^{\ep^{-1}h}du\Big|^2dh\nonumber\\&\lesssim&\mathbb{E}\int_{0}^{\tau}\Big|\int_{\tau}^{\tau+\delta}\ep^{\theta-1}e^{-\ep^{-1}u}e^{\ep^{-1}h}du\Big|^2dh\nonumber\\&\leq&\mathbb{E}\int_{0}^{\tau}\Big(\int_{\tau}^{\tau+\delta}\ep^{2\theta-2}e^{-2\ep^{-1}u}e^{2\ep^{-1}h}du\Big)\delta dh\nonumber\\&=&\delta\ep^{2\theta-2}\mathbb{E}\int_{0}^{\tau}\int_{\tau}^{\tau+\delta}e^{-2\ep^{-1}(u-h)}dudh.
		\end{eqnarray}
		A direct calculation yields
		$$\int_{0}^{\tau}\int_{\tau}^{\tau+\delta}e^{-2\ep^{-1}(u-h)}dudh\lesssim\ep^2,$$ so that
		\begin{equation}\label{tightJ21k}
			\mathbb{E}|J^{21}_k|^2\lesssim\delta\ep^{2\theta}\leq\delta.
		\end{equation}
		For $J^{22}_k,$ by It\^{o} isometry,
		\begin{eqnarray}
			&&\mathbb{E}|J^{22}_k|^2\nonumber\\&=&\mathbb{E}|\int_{\tau}^{\tau+\delta}\int_{h}^{\tau+\delta}\ep^{\theta-1}e^{-\ep^{-1}u}e^{\ep^{-1}h}du\tilde{L}_k(h)|^2\nonumber\\&=&\mathbb{E}\int_{\tau}^{\tau+\delta}\int_{|x|< 1}\Big|\int_{h}^{\tau+\delta}x\ep^{\theta-1}e^{-\ep^{-1}u}e^{\ep^{-1}h}du\Big|^2\nu_1(dx)dh\nonumber\\&=&\int_{|x|< 1}|x|^2\nu_1(dx)\mathbb{E}\int_{\tau}^{\tau+\delta}\Big|\int_{h}^{\tau+\delta}\ep^{\theta-1}e^{-\ep^{-1}u}e^{\ep^{-1}h}du\Big|^2dh.
		\end{eqnarray}
		By a direct calculation,$$\int_{h}^{\tau+\delta}\ep^{\theta-1}e^{-\ep^{-1}u}e^{\ep^{-1}h}du\leq\ep^\theta,$$so
		$$\mathbb{E}|J^{22}_k|^2\lesssim\int_{|x|< 1}|x|^2\nu_1(dx)\ep^{2\theta}\delta\lesssim\delta.$$
		Together with (\ref{tightJ21k}),
		\begin{equation*}
			\mathbb{E}|J^2_k(\tau+\delta)-J^2_k(\tau)|^2\lesssim\delta.
		\end{equation*}
		By Jensen inequality,
		\begin{equation}\label{tightJ2k}
			\mathbb{E}|J^2_k(\tau+\delta)-J^2_k(\tau)|\lesssim\sqrt{\delta}.
		\end{equation}
		Substituting (\ref{tightJ1k}) and (\ref{tightJ2k}) into (\ref{tightJkready}),
		\begin{equation}\label{tightJk}
			\mathbb{E}|J^\ep_k(\tau+\delta)-J^\ep_k(\tau)|\lesssim\sqrt{\delta},
		\end{equation}
		from which we derive by (\ref{tightI3ready})
		\begin{equation}\label{tightI3}
			\mathbb{E}|I^\ep_3(\tau+\delta)-I^\ep_3(\tau)|\lesssim\sqrt{\delta}.
		\end{equation}
		The proof is finished by combining (\ref{tightI1}), (\ref{tightI2}) and (\ref{tightI3}).
	\end{proof}
	\section{Proof of the Main Result}\label{proofofmain}
	This section devotes to proving Theorem \ref{mainresult}. We begin with treating velocity part $\vbe$.
	\begin{lemma}
		$\mathbb{E}\supt|\ep^{-1}\intot\bvbe_1(s)ds|\lesssim \ep.$\label{esforbv1}
	\end{lemma}
	\begin{proof}
		From (\ref{bveq})  
		\begin{equation}
			\bvbe_1(t)=\ep v_0 e^{-\ep^{-1}t}.\label{precisebv1}
		\end{equation}
		As a consequence, for all $0\leq t\leq T$\,, 
		\begin{eqnarray}
			\Big|\ep^{-1}\intot \bvbe_1(s)ds\Big|&=&\Big|\intot v_0 e^{-\ep^{-1}s}ds\Big|\nonumber\\
			&\leq&|v_0|\int_{0}^{T}e^{-\ep^{-1}s}ds\nonumber\\
			&\leq &\ep|v_0|.
		\end{eqnarray}
		Taking supremum and expectation yields the result. 
	\end{proof}
	
	\begin{lemma}\label{esforbv2}
		$$\supe\mathbb{E}\supt|\bvbe_2(t)|<\infty$$
	\end{lemma}	
	\begin{proof}
		From (\ref{bveq}),
		\begin{equation}
			\bvbe_2(t)=\ep^{-1}e^{-\ep^{-1}t}\intot e^{\ep^{-1} s}f(\ube(s))ds.\label{precisebv2}
		\end{equation}
		Thanks to the Lipschitz continuity of $f$, for all $0\leq t\leq T$
		\begin{eqnarray}
			&&|\bvbe_2(t)|\nonumber\\&\lesssim&\ep^{-1}e^{-\ep^{-1}t}\intot e^{\ep^{-1} s}|\ube(s)|ds+\ep^{-1}e^{-\ep^{-1}t}\intot e^{\ep^{-1} s}ds\nonumber\\&\leq&\ep^{-1}e^{-\ep^{-1}t}\supt|\ube(t)|\intot e^{\ep^{-1} s}ds+1\nonumber\\&\leq &\supt|\ube(t)|+1,
		\end{eqnarray}
		and the result follows from Proposition \ref{thmforEsup} after taking expectation.
		
	\end{proof}
	
	We also need a lemma which states that the convergence in Skorokhod topology is stronger than $L^1$ convergence.
	\begin{lemma}\label{skorokhod}
		Let $(x_n)$ be a sequence in $\mathbb{D}$. If $\lim\limits_{n\to\infty}d^o(x_n,x)=0,$ then $x_n\to x$ in $L^1([0,T];\mathbb{R}^d).$
	\end{lemma}
	\begin{proof}
		For $x,y\in\mathbb{D},$ define $$d(x,y):=\inf\limits_{\lambda\in\Lambda}\max\{||\lambda-id||_\infty,||x-y\circ\lambda||\},$$
		where $id$ is the identity map on $[0,T]$ and $\Lambda$ is defined in Section \ref{preli}. By \cite[Theorem 12.1]{BIL13}, $d^o$ and $d$ are equivalent, so $\lim\limits_{n\to\infty}d(x_n,x)=0$, and by the definition of $d$, there exists a sequence $(\lambda_n)$ in $\Lambda$, such that 
		\begin{equation}\label{resultofconvergence}
			\lim_{n\to\infty}||x_n-x\circ\lambda_n||_\infty=0,\quad\lim_{n\to\infty}||\lambda_n-id||_\infty=0.
		\end{equation}
		In particular, 
		\begin{equation*}
			\sup_{n\in\mathbb{N}}\supt|x_n(t)-x(\lambda_n(t))|=\sup_{n\in\mathbb{N}}||x_n-x\circ\lambda_n||_\infty<\infty.
		\end{equation*}
		By bounded convergence theorem and (\ref{resultofconvergence}),
		\begin{equation}\label{L1convergence1}
			\lim_{n\to\infty}\int_0^T|x_n(t)-x\circ\lambda_n(t)|dt=0.
		\end{equation}
		Since $x$ is c\`adl\`ag, it has at most countable discontinuity, so that by (\ref{resultofconvergence}),
		\begin{equation}\label{integrandconvergence}
			\lim_{n\to\infty}|x(\lambda_n(t))-x(t)|=0,\quad\text{a.e.}.
		\end{equation}
		Using the c\`adl\`ag property again, $x$ is bounded, so by bounded convergence theorem and (\ref{integrandconvergence}),
		\begin{equation}\label{L1convergence2}
			\lim_{n\to\infty}\int_{0}^{T}|x\circ\lambda_n(t)-x(t)|dt=0,
		\end{equation}
		from which the lemma follows from noticing that
		$$|x_n(t)-x(t)|\leq|x_n(t)-x\circ\lambda_n(t)|+|x\circ\lambda_n(t)-x(t)|$$ and (\ref{L1convergence1}).
	\end{proof}

	With the preparation made above, we are in a position to prove our main result. From (\ref{slowfast}) and~(\ref{decomposition}) we have 
	\begin{equation}
		\ube(t)=u_0+\ep^{-1}\intot\bvbe_1(s)ds+\intot\bvbe_2(s)ds+\ep^{\theta+\frac{1}{\al}-1}\intot\bvbe_3(s)ds.
	\end{equation}
	By (\ref{bveq}), 
	\begin{equation*}
		\bvbe_2(t)=\intot -\ep^{-1}[\bvbe_2(s)-f(\ube(s))]ds.
	\end{equation*}
	Combining the two equations above,
	\begin{equation}
		\ube(t)=u_0+\ep^{-1}\intot\bvbe_1(s)ds+\intot f(\ube(s))ds-\ep\bvbe_2(t)+\ep^{\theta+\frac{1}{\al}-1}\intot\bvbe_3(s)ds.\label{newequationforu}
	\end{equation}
	From (\ref{aveq}) and (\ref{newequationforu}) we deduce that 
	\begin{eqnarray}
 &&|\ube(t)-\bube(t)|\nonumber\\ &\leq& \Big|\ep^{-1}\intot\bvbe_1(s)ds\Big|+\Big|\intot f(\ube(s))-f(\bube(s))ds\Big|+\ep|\bvbe_2(t)|+\Big|\ep^{\theta+\frac{1}{\al}-1}\intot\bvbe_3(s)ds-\ep^\al L(t)\Big| \nonumber\\&=:& \sum_{k=1}^{4}I^\ep_k(t). \label{generaldifference}
	\end{eqnarray}
	As a consequence of Lemma \ref{esforbv1}, 
	\begin{equation}
		\mathbb{E}\supt I^\ep_1(t)\lesssim\ep. \label{esforI1}
	\end{equation}
	Due to the fact that $f$ is Lipschitz, for all $0\leq t\leq T$,
	\begin{eqnarray*}
			&&\Big|\intot f(\ube(s))-f(\bube(s))ds\Big|\\ &\leq &\intot |f(\ube(s))-f(\bube(s))|ds\\&\leq&\int_{0}^{T} \sup\limits_{0\leq r\leq s}|f(\ube(r))-f(\bube(r))|ds\\ &\lesssim &\int_{0}^{T} \sup\limits_{0\leq r\leq s}|\ube(r)-\bube(r)|ds.
	\end{eqnarray*}
	Taking supremum and expectation and using Fubini theorem,
	\begin{equation}
		\mathbb{E}\supt I^\ep_2(t)\lesssim \int_{0}^{T}\mathbb{E}\sup\limits_{0\leq s\leq t}|\ube(s)-\bube(s)|dt. \label{esforI2}
	\end{equation}
	From Lemma \ref{esforbv2},
	\begin{equation}\label{esforI3}
		\mathbb{E}\supt I^\ep_3(t)\lesssim\ep.
	\end{equation}
	Lastly we deal with $I^\ep_4$\,. From (\ref{bveq}), 
	\begin{equation}
		\ep^{\frac{1}{\al}}\bvbe_3(t)=L(t)-\ep^{\frac{1}{\al}-1}\intot \bvbe_3(s)ds,\label{newI3}
	\end{equation}
	but one finds that $\bvbe_3$ coincides with $\we$ defined in (\ref{defofw}), so that $\ep^\frac{1}{\al}\bvbe_3$ coincides with $\bwe.$ Therefore, (\ref{esforbwe}) implies that there exists a constant $C>0$ such that for all $0<\ep\leq 1$, 
	$$\mathbb{E}\supt|L(t)-\ep^{\frac{1}{\al}-1}\intot \bvbe_3(s)ds|\leq C.$$
	Multiplying both sides by $\ep^\theta$,
	\begin{equation}
		\mathbb{E}\supt I^\ep_4(t)=\mathbb{E}\supt|\ep^\theta L(t)-\ep^{\theta+\frac{1}{\al}-1}\intot \bvbe_3(s)ds|\leq C\ep^\theta.\label{esforI4}
	\end{equation}
	Taking supremum and expectation on both sides of (\ref{generaldifference}) and combining (\ref{esforI1})--(\ref{esforI4}), 
	\begin{eqnarray}
		&&\mathbb{E}\supt|\ube(t)-\bube(t)|\nonumber\\ &\lesssim& \int_{0}^{T}\mathbb{E}\sup\limits_{0\leq s\leq t}|\ube(s)-\bube(s)|dt+\ep+\ep^\theta\nonumber\\ &\leq&\int_{0}^{T}\mathbb{E}\sup\limits_{0\leq s\leq t}|\ube(s)-\bube(s)|dt+\ep^\theta,
	\end{eqnarray}
	where the last inequality follows from the fact that $0\leq\theta<1.$
	By Gronwall inequality,
	$$\mathbb{E}\supt|\ube(t)-\bube(t)|\lesssim \ep^\theta,$$
	and we finish the proof for part (i) of Theorem \ref{mainresult}.
	
	Next we turn to the proof of part (ii) of Theorem \ref{mainresult}. Firstly we show that $(\ube)$ converges in probability as $\ep\to 0$. By the classical result in \cite{GK96}, it is sufficient to prove that for any two subsequences $\{\ep(n)\}_{n\in\mathbb{N}}$ and $\{\mu(n)\}_{n\in\mathbb{N}}$ with $\ep(n)\to 0$ and $\mu(n)\to0,$ there exist subsequences $\{\ep(n_k)\}$ and $\{\mu(n_k)\}$ such that $(U^{\ep(n_k)},U^{\mu(n_k)})$ converges weakly to a $\mathbb{D}^2$-valued random variable $w=(w_1,w_2)$, and $w$ supports on the diagonal.
	
	Now given two subsequences $\{\ep(n)\}$ and $\{\mu(n)\}$ with $\ep(n)\to 0$ and $\mu(n)\to0,$ since $(\ube)_{0<\ep\leq 1}$ is  tight, by Prokhorov and Skorokhod theorem, there exist
	
	(i) a probability space $(\hat{\Omega},\hat{\mathcal{F}},\hat{\mathbb{P}}),$
	
	(ii) a sequence of $\mathbb{D}^3$-valued random variable $(u_1^k,u_2^k,\hat{L}_k)$ defined on $(\hat{\Omega},\hat{\mathcal{F}},\hat{\mathbb{P}}),$
	
	(iii) a $\mathbb{D}^3$-valued random variable $(u_1,u_2,\hat{L})$ defined on $(\hat{\Omega},\hat{\mathcal{F}},\hat{\mathbb{P}}),$
	
	\noindent such that 
	
	(i) $(u_1^k,u_2^k,\hat{L}_k)\overset{d}{=}(U^{\ep(n_k)},U^{\mu(n_k)},L),$
	
	(ii) $(u_1^k,u_2^k,\hat{L}_k)\to (u_1,u_2,\hat{L})$, $\hat{\mathbb{P}}-$ a.s..
	
	Define
	\begin{equation}\label{defofR}
		R^\ep(t):=\ube(t)-u_0-\intot f(\ube(s))ds-L(t),
	\end{equation}
	then by (\ref{newequationforu}),
	\begin{eqnarray}\label{newequationforR}
		&&R^\ep(t)\nonumber\\&=&\ep^{-1}\intot\bvbe_1(s)ds-\ep\bvbe_2(t)+\Big(\ep^{\frac{1}{\al}-1}\intot\bvbe_3(s)ds-L(t)\Big).
	\end{eqnarray}
	(\ref{esforI1}) and (\ref{esforI3}) indicate that
	$$\mathbb{E}\Big|\ep^{-1}\intot\bvbe_1(s)ds\Big|\lesssim\ep,$$
	and
	$$\ep\mathbb{E}|\bvbe_2(t)|\lesssim\ep,$$
	respectively. As we have pointed out in the context of (\ref{newI3}),
	$$\Big|\ep^{\frac{1}{\al}-1}\intot\bvbe_3(s)ds-L(t)\Big|=|\ep^{\frac{1}{\al}}\bvbe_3(t)|,$$
	and that $\ep^{\frac{1}{\al}}\bvbe_3(t)=\bwe(t)$ with
	$$d\bwe(t)=-\ep^{-1}\bwe(t)+dL(t),\quad \bwe(0)=0,$$
	so it is sufficient to estimate $\mathbb{E}|\bwe(t)|.$ Since
	\begin{equation}\label{precisebwe}
		 \bwe(t)=\intot e^{-\ep^{-1}(t-s)}dL(s),
	\end{equation}
	we have by \cite[Theorem 3.2]{RW06}
	\begin{eqnarray*}
		&&\mathbb{E}|\bwe(t)|\\&\lesssim&\Big(\intot|e^{-\ep^{-1}(t-s)}|^\al ds\Big)^{1/\al}\\&=&\Big(\intot|e^{-\ep^{-1}s}|^\al ds\Big)^{1/\al}\\&=&\Big[\frac{\ep}{\al}(1-e^{-\ep^{-1}\al t})\Big]^{1/\al},
	\end{eqnarray*}
	so that
	\begin{eqnarray*}
		\lim_{\ep\to0}\mathbb{E}\Big|\ep^{\frac{1}{\al}-1}\intot\bvbe_3(s)ds-L(t)\Big|=\lim_{\ep\to0}\mathbb{E}(\ep^\frac{1}{\al}|\bvbe_3(t)|)=\lim_{\ep\to0}\mathbb{E}|\bwe(t)|=0.
	\end{eqnarray*}
	The estimates above imply that for each $0\leq t\leq T,$
	\begin{equation}
		\lim_{\ep\to0}\mathbb{E}|R^\ep(t)|=0.
	\end{equation}
	As in (\ref{defofR}), set for $i=1,2$ and $k\in\mathbb{N},$ 
	$$R^k_i(t):=u^k_i(t)-u_0-\intot f(u^k_i(s))ds-\hat{L}_k(t).$$
	Since $(u_1^k,u_2^k,\hat{L}_k)\overset{d}{=}(U^{\ep(n_k)},U^{\mu(n_k)},L),$ denote $\hat{\mathbb{E}}$ to be the expectation with respect to $\hat{\mathbb{P}},$
	$$\lim\limits_{k\to\infty}\hat{\mathbb{E}}|R^k_i(t)|=0,$$ so there is a subsequence converges to $0$ almost surely , and we still denote it by $R^k_i(t)$.
	
	Set $\mathcal{D}_i:=\{t\in[0,T]:P(u_i(t)\neq u_i(t-))>0\}$ and $\mathcal{D}:=\mathcal{D}_1\cup\mathcal{D}_2$. From the c\`{a}dl\`ag property of $u_i$, $\mathcal{D}$ is at most countable\cite[Lemma 7.7 in Chapter 3]{EK09}. Since $u^k_i\to u_i$ in Skorokhod topology almost surely and $f$ is Lipschitz, it is straightforward to check that, for each $t\notin\mathcal{D},$ $u^k_i(t)\to u_i(t),$ and that $f(u^k_i)\to f(u_i)$ in Skorokhod topology almost surely. On applying Lemma \ref{skorokhod}, we conclude that
	\begin{equation}\label{heatequation}
		u_i(t)-u_0-\intot f(u_i(s))ds-\hat{L}(t)=0,
	\end{equation}
	$i=1,2$ for each $t\notin\mathcal{D},$ $\hat{\mathbb{P}}-a.s..$ The c\`adl\`ag property impies that the above equality holds for all $0\leq t\leq T,$ $\hat{\mathbb{P}}-a.s.,$ so $u_1=u_2$ due to the pathwise uniqueness of the first-order SDE. Therefore, $(\ube)$ converges in probability to some $\bar{U}$ in $\mathbb{D}$\cite{GK96}. But by the same argument we derive that $\bar{U}$ again satisfies (\ref{heatequation}). The proof is finished. \qed\\
	
	We show that the stronger convergence $$\lim\limits_{\ep\rightarrow 0}\mathbb{E}\supt|\ube(t)-\bar{U}(t)|=0,$$ is not true. Indeed, as we point out in Remark \ref{nottightinC}, $\ube\in\mathbb{C}$ for each $0<\ep\leq1.$ By contradictory, suppose that the stronger convergence above is true, then there exists a subsequence $(U^{\ep(k)})$ such that $$\lim\limits_{k\rightarrow \infty}\supt|U^{\ep(k)}(t)-\bar{U}(t)|=0,\quad\mathbb{P}-\text{a.s..}$$ In particular, $\bar{U}$ is also continuous. But we showed in Remark \ref{nottightinC} that $\bar{U}$ is not continuous in general. 
	
	Now we are ready to give a corollary, which is independent of our topic, but whose proof seems not obvious without the discussion above.
	\begin{coro}
		For the Ornstein-Uhlenbeck process
		\begin{equation}\label{OU}
			Z^\ep(t):=\intot e^{-\ep^{-1}(t-s)}dL(s),
		\end{equation}
		or equivalently,
		\begin{equation}\label{OU2}
			dZ^\ep(t)=-\ep^{-1}Z^\ep(t)+dL(t),\quad Z^\ep(0)=0,
		\end{equation}
		we have
		\begin{equation}
			\limsup_{\ep\to0}\mathbb{E}\supt|Z^\ep(t)|>0.
		\end{equation}
	\end{coro}
	\begin{proof}
		From (\ref{OU2}) we see that $$Z^\ep=\ep^{\frac{1}{\al}}\bvbe_3.$$ From (\ref{generaldifference})--(\ref{newI3}), 
		$$\mathbb{E}\supt|\ube(t)-\bube(t)|\lesssim \int_{0}^{T}\mathbb{E}\sup\limits_{0\leq s\leq t}|\ube(s)-\bube(s)|dt+\ep+\mathbb{E}\supt|\ep^{\frac{1}{\al}}\bvbe_3(t)|\,.$$
		Then we have
		\begin{equation}\label{stronger}
			\lim\limits_{\ep\rightarrow 0}\mathbb{E}\supt|\ube(t)-\bube(t)|=0
		\end{equation}
		  by Gronwall inequality, provided that
		\begin{equation}
			\lim\limits_{\ep\rightarrow 0}\mathbb{E}\supt|\ep^{\frac{1}{\al}}\bvbe_3(t)|=0.\label{star}
		\end{equation}
		However, we know that (\ref{stronger}) is not true, and so is (\ref{star}). 
	\end{proof}

	\bibliographystyle{abbrv}
	\addcontentsline{toc}{section}{References}
	\bibliography{bibli}	
	
\end{document}